\documentclass[letterpaper, 10 pt, conference]{ieeeconf}  

\IEEEoverridecommandlockouts                              
\usepackage[T1]{fontenc}
\usepackage{textcomp}
\usepackage{eso-pic}

\usepackage{amsmath, amssymb}
\usepackage{cite}
\usepackage{algorithm}
\usepackage{algpseudocode}
\usepackage{graphicx}
\usepackage{bm, xcolor}
\usepackage[normalem]{ulem}
\newtheorem{theorem}{Theorem}
\newtheorem{proposition}{Proposition}

\title{\LARGE \bf
Moral Hazard in LTI Dynamics: A Hypothesis Testing Approach}

\author{Jaewon Jeong, Pan-Yang Su, S. Shankar Sastry, and Anil Aswani
\thanks{This work was partially supported by Collaborative Research: Transferable, Hierarchical, Expressive, Optimal, Robust, Interpretable NETworks (THEORINET) under Simons Foundation Award MPS-MODL-00814647 and National Science Foundation (NSF) Grant DMS-2031899, and by the NSF under Grant DGE-2125913.}
\thanks{J. Jeong and A. Aswani are with Industrial Engineering and Operations Research,
        University of California, Berkeley, CA 94720
        {\tt\small \{jaewon\_jeong,aaswani\}@berkeley.edu}}%
\thanks{P.Y. Su and S.S. Sastry are with the Electrical Engineering and Computer Sciences, University of California, Berkeley, CA 94720
        {\tt\small \{pan\_yang\_su,shankar\_sastry\}@berkeley.edu}}%
}

\begin{document}

\maketitle
\AddToShipoutPictureFG*{%
\put(\LenToUnit{0.5\paperwidth},\LenToUnit{0.55in}){%
\makebox(0,0){\parbox{0.85\paperwidth}{\centering\scriptsize
Accepted to the 2026 IEEE Conference on Decision and Control (CDC).\\
\copyright~2026 IEEE. Personal use of this material is permitted.
Permission from IEEE must be obtained for all other uses, in any
current or future media, including reprinting/republishing this
material for advertising or promotional purposes, creating new
collective works, for resale or redistribution to servers or lists,
or reuse of any copyrighted component of this work in other works.}}}%
}
\thispagestyle{empty}
\pagestyle{empty}

\begin{abstract}
Many incentive design problems must contend with information asymmetries due to non-observation of efficiency (adverse selection) or non-observation of effort (moral hazard). Although a growing body of literature considers incentive design in control systems, the problem of designing incentives for control systems under information asymmetries has been less studied. This paper considers a model of moral hazard within control systems. In our model, the control system is described by an (affine) linear time-invariant (LTI) system with process noise. There is an agent who gets to choose (from between two choices) a linear state-feedback controller to apply to the LTI system, with one of the state-feedback controllers having a higher quadratic cost on the control inputs than the other. Our goal is to design a payment scheme that incentivizes the agent to choose the state-feedback controller that minimizes a quadratic cost on system states plus the time-discounted payment amount, subject to the understanding that the agent bears the control cost while being risk-averse with respect to their time-discounted payment. We formulate the problem as a constrained optimization, and prove that for a  payment given after a fixed (but optimizable) time horizon the optimal payment scheme chooses the payment amount using a likelihood ratio hypothesis test. We numerically demonstrate our results by applying the derived optimal payment scheme to two examples: load frequency control (LFC) in power systems and wellness interventions for body weight loss.
\end{abstract}

\section{INTRODUCTION}
Economic aspects of control systems are increasingly important given the growth of applications at the intersection of economics and control, like power \cite{Tan2010,5718175} or transportation \cite{su2024incentive,11298540} systems. As such, an active body of literature within controls considers various problems of incentive design, in which a \emph{principal} would like to design a scheme for payments to an \emph{agent} to encourage the agent to take actions that are favorable for the principal. A recurring obstacle in these problems is \emph{moral hazard}: the principal cannot observe the agent's control effort directly, but only a noisy trajectory of system outputs that the effort influences. This arises, for instance, when a system operator cannot inspect a frequency regulation provider's effort level, or when a health insurer cannot directly observe a patient's adherence to a physical-activity regimen and instead sees only intermittently measured outcomes. In both cases, the principal must infer effort from data and design payments accordingly, and doing so requires an explicit statistical procedure for turning an observed trajectory into a payment decision. 

By restricting attention to two effort levels and two payment levels, we formulate the principal's decision as a binary hypothesis test of whether the agent has exerted high or low effort. The output trajectory serves as an informative signal about the agent's unobservable effort, and the likelihood ratio naturally measures strength of the evidence in favor of high versus low effort, making it natural to assign a higher payment to such realizations. The Neyman-Pearson lemma formalizes this intuition by showing that the optimal payment region is characterized by a likelihood-ratio test. In short, we propose a principal-agent model for stochastic LTI systems with continuous outputs subject to moral hazard. For the class of contracts that has two payment levels, we show that a likelihood ratio test on the observed output trajectory is the optimal decision rule and can be computed in closed form. 
\subsection{Related Work}
Much of the literature on incentive design for control systems focuses on the setting of perfect information, an early example of which is \cite{HO1982167}. In these models, the principal is assumed to observe the agent's control actions directly, so the design problem reduces to computing incentive-compatible payments without any inference step. This line of work has extended to settings such as bilevel model predictive control (BiMPC) \cite{mintz2018control,mintz2023behavioral}, and has recently been generalized to more complex hierarchical MPC structures \cite{thirugnanam2025dynamic}. 

More recently, researchers have begun to study incentive design for control systems under information asymmetry, including dynamic contracts for stochastic control systems that address moral hazard \cite{Yang2014}, repeated principal-agent games in which the agent's rewards are learned online via bandit feedback \cite{dogan2023repeated,dogan2023estimating}, and adaptive incentive design under adverse selection \cite{9209172,vasileiou2026adaptiveincentivedesignregret}, to name just a few; see also \cite{ratliff2019perspective} for a survey.

The work closest to our model is \cite{Yang2014}, which designs dynamic contracts for stochastic control systems by solving a Hamilton-Jacobi-Bellman (HJB) equation for continuous-time, nonlinear dynamics. In contrast, we consider discrete-time LTI dynamics with linear state feedback and a simple two-payment-level contract. While a continuous-time setting allows generality to accommodate a broad class of systems, a discrete-time setup is sufficient for our purposes and naturally aligns with real-world applications, as digital controllers, sensors, and financial accounting systems operate in discrete time \cite{franklin1998digital, pliska1997introduction}. Also, by focusing on simple LTI dynamics, we achieve closed-form expressions for optimal payment. This yields a more interpretable contract and, under Gaussian noise, allows us to characterize the decision rule's error probabilities exactly (Theorem~\ref{thm:distribution}), rather than asymptotically or through simulation.

\subsection{Contribution and Outline}
Our first contribution concerns the structure of the optimal contract within this restricted class: for a moral hazard model with stochastic discrete-time LTI dynamics and a two-payment-level contract, we show the optimal decision region is characterized by a likelihood ratio test (Section~\ref{sec:soc}). This is an application of the generalized Neyman-Pearson lemma to our contracting setup; similar reductions are known when outputs are discrete \cite{Laffont,saig2023delegated}. 

Our second contribution is to show that the distribution of the log-likelihood ratio statistic under stochastic LTI dynamics takes the form of a generalized chi-squared distribution (Theorem~\ref{thm:distribution}). Along with the discrete-time setting, this result enables us to reduce the infinite-dimensional contract design problem to two nested line searches (Section~\ref{sec:ctoc}), which allows us to explicitly characterize the optimal contract for continuous-valued system outputs.

Finally, we demonstrate our theoretical results on two applications: load frequency control in power systems, where the principal is a system operator inferring a generator's internal control settings, and a wellness incentive program, where the principal is a health insurer inferring a patient's activity level from weekly weight measurements (Section~\ref{sec:ae}). 

The remainder of this paper is organized as follows. Section \ref{sec:cdp} presents the system model and the contract design problem. In Section \ref{sec:soc}, we theoretically characterize the structure of the optimal contract, and Section \ref{sec:ctoc} provides a practical numerical algorithm to compute the optimal contract. Lastly, Section \ref{sec:ae} presents two real-world applications. 

\section{CONTRACT DESIGN PROBLEM}
\label{sec:cdp}
We consider the following problem. A \emph{principal} wishes to control a dynamical system by paying an \emph{agent} to choose control actions on their behalf. The principal incurs a state cost, and the agent incurs a control cost. The agent picks a controller from a choice of two state-feedback policies, with one of the state-feedback policies having a higher expected control cost than the other. The principal's goal is to design a contract that maps any state trajectory to a payment amount, where the payment serves as a cost to the principal and a reward to the agent. 

\subsection{Control System and Costs}
\subsubsection{System Dynamics}
The system the principal would like to control is described by a discrete-time linear time-invariant (LTI) system with process and measurement noise
\begin{align}
    x_{k+1}&=Ax_{k}+Bu_{k}+w_{k}\\
    y_k &= Cx_k + \epsilon_k
\end{align}
where $x_{k} \in \mathbb{R}^{n}$ is the system state, $y_k \in \mathbb{R}^p$ is the noisy output, $u_k \in \mathbb{R}^m$ is the control input, the matrices have dimensions $A\in\mathbb{R}^{n\times n}$, $B\in\mathbb{R}^{n\times m}$, $C\in\mathbb{R}^{p\times n}$, $w_{k} \sim \mathcal{N}(\mu_{w}, \Sigma_{w})$ with $\mu_w \in \mathbb{R}^n$ and $\Sigma_w \in S^n_+$ (with $S^n_+$ denoting the set of $\mathbb{R}^{n\times n}$ positive semidefinite matrices) is independent and identically distributed (i.i.d.) jointly Gaussian process noise, and $\epsilon_{k} \sim \mathcal{N}(0, \Sigma_{e})$ with $\epsilon_k\in\mathbb{R}^p$ and $\Sigma_e \in S^p_+$ is i.i.d. jointly Gaussian measurement noise. Our model allows the process noise to have non-zero mean $\mu_w$ because this allows us to include affine LTI systems in our results. The initial state $x_0$ has a jointly Gaussian distribution $x_0 \sim \mathcal{N}(\mu_0, \Sigma_{0})$ with $\mu_0 \in \mathbb{R}^n$ and $\Sigma_0 \in S^n_+$. We further assume that $x_0$, $w_t$, $\epsilon_t$ for all $t$ are independent.

In our model, we do \emph{not} require $(A,C)$ to be observable because such an assumption is stronger than needed for our problem. We will instead make a related but weaker assumption about the distributions of the outputs $y_0,y_1,\ldots$. This assumption will appear in the statements of our theorems. 

\subsubsection{Agent Costs}

The agent chooses between two linear state-feedback controllers, one $u_k = K_0 x_k$ which is interpreted as a \emph{low effort} controller, and the other $u_k = K_1 x_k$ which is interpreted as a \emph{high effort} controller. For the matrices $K_0, K_1 \in \mathbb{R}^{m\times p}$, we assume  that $(A+BK_0)$ and $(A+BK_1)$ are Schur stable. The interpretation of controller effort levels is driven by a time-discounted quadratic control cost
\begin{equation}    
    J^A = \mathbb{E}\Big[ \textstyle\lim_{N \to \infty} \sum_{t=0}^{N}\gamma_a^t \cdot(u_t^\top R u_t + r^\top u_t)\Big],
\end{equation}
where $R \in \mathbb{R}^{m\times m}$ and $r \in \mathbb{R}^m$ are constants of the quadratic cost, and $\gamma_a \in (0,1)$ is a discount factor. Let $J^A_0$ be the value of this cost under the low effort controller $K_0$, and let $J^A_1$ be the value of this cost under the high effort controller $K_1$. We assume the values of the various constant matrices and vectors  in our model are such that $J^A_1 > J^A_0$ to ensure the high versus low effort interpretation of $K_1$ versus $K_0$.

The agent also receives a payment $\pi$ at time $T$ from the principal as an incentive, and the agent is assumed to be risk-averse with respect to this payment amount. We model the time-discounted utility of the payment as $\gamma_a^T\cdot U(\pi)$, where $U: \mathbb{R}\rightarrow\mathbb{R}$ is a mapping from a payment amount to a utility. We assume that $U$ satisfies the following properties, which are standard in risk-aversion models \cite{Laffont}: $U(0) = 0$, $U' > 0$, and $U'' < 0$. 
Hence, the agent's total cost is given by $J^A - \gamma_a^T \cdot U(\pi)$. Note that the minus sign implies that positive utility reduces the agent's cost.

\subsubsection{Principal Costs}

The principal incurs a time-discounted state cost
\begin{equation}
    J^P = \mathbb{E}\Big[ \textstyle\lim_{N \to \infty} \sum_{t=0}^{N}\gamma_p^t \cdot(x_t^\top Q x_t + q^\top x_t)\Big],
\end{equation}
where $Q \in \mathbb{R}^{p\times p}$ and $q \in \mathbb{R}^p$ are constants of the cost, and $\gamma_p \in (0,1)$ is a discount factor. Let $J^P_0$ be the value of this cost under the low effort controller $K_0$, and let $J^P_1$ be the value of this cost under the high effort controller $K_1$. 

The principal makes a payment $\pi$ at time $T$ to the agent as an incentive, and the principal incurs a cost of $\gamma_p^T \cdot \pi$ for this payment. Hence, the principal's total cost is $J^P + \gamma_p^T \cdot \pi$. Note that $\pi$ is a function, and the choice of $\pi$ and $T$ will be formulated as an optimization problem in Section \ref{subsec: contract design problem}.


\subsection{Contract Design Problem}
\label{subsec: contract design problem}
In our setup, the principal must make two choices. The first choice is to decide whether or not they would like the agent to exert high effort via a payment. This is a choice because, depending on the coefficient values, the payment cost required to induce the agent to exert high effort may exceed the cost reduction the principal gains when the agent exerts high effort. The second choice is that if the principal decides it is beneficial to induce high effort from the agent, then the principal must choose a contract to offer the agent. However, the problem of designing the optimal contract in moral hazard settings typically proceeds backwards \cite{Laffont}: The principal first computes an optimal contract $\pi^*$ given an affirmative decision to induce high effort, and then the principal checks whether it is beneficial to execute this contract to induce high effort. 

In our setting, moral hazard arises because the principal does not directly observe the effort level (i.e., the choice of $K_0$ or $K_1$) of the agent. Instead, we assume the principal takes measurements $y_0, y_1, \ldots$. Since the mapping from the agent's choice of effort level to the outputs is probabilistic due to the process and measurement noise, it is impossible to perfectly identify the agent's choice of effort level from the observations. Therefore, the contract design problem must necessarily be probabilistic.

Before we present the optimal contract design problem, we first formally specify the structure of the class of contracts we consider. Let $\mathbf{Y}_T = (y_1,\ldots,y_T)$. Then a contract is a function $\pi: \mathbf{Y}_T \mapsto \pi(\mathbf{Y}_T) \in \mathbb{R}$ that maps the observed outputs $\mathbf{Y}_T = (y_1,\ldots,y_T)$ over a finite horizon $T$ to a payment $\pi(\mathbf{Y}_T)$, which may be positive or negative. 
For the purpose of practical implementation, it is important for the contract to be easily interpretable by the agent. Consequently, here we further restrict the contract to have the following easily-interpretable structure:
\begin{equation}
\pi(\mathbf{Y}_T) = \begin{cases} \pi_0, & \text{if } \mathbf{Y}_T \notin \mathcal{R} \\ \pi_1, & \text{if } \mathbf{Y}_T \in \mathcal{R} 
\end{cases}
\end{equation}
where the values $T \in\mathbb{Z}_{+}$ and $\pi_0,\pi_1 \in \mathbb{R}$ and the set $\mathcal{R}$ are to be optimized in the contract design problem. The interpretation is that $\mathcal{R}$ is a region for which the state trajectory is more likely to belong to under high effort, whereas the complement is a region for which the state trajectory is more likely to belong to under low effort. The contract then pays $\pi_0$ for low effort, and $\pi_1$ for high effort.

It is now straightforward to formulate the contract design problem assuming the principal wishes to induce high effort. Let $H_1$ be the distributions of $y_0,y_1,\ldots$ that result from our model when the agent chooses high effort $K_1$, and let $H_0$ be the distributions of $y_0,y_1,\ldots$ that result from our model when the agent chooses low effort $K_0$. We assume that $H_0 \neq H_1$. For notational convenience, we define $\alpha(\mathcal{R}) = \mathbb{P}[\mathbf{Y}_T \in \mathcal{R}\ |\ H_0]$ and $\beta(\mathcal{R}) = \mathbb{P}[\mathbf{Y}_T \in \mathcal{R}\ |\ H_1]$. We will drop the argument and write $\alpha,\beta$ when the meaning is clear from the context. The contract design problem is 
\begin{align}
     \min\ & J_1^P + \gamma_p^T \cdot \big(\pi_0 \cdot (1-\beta) + \pi_1 \cdot \beta\big) \label{eqn:cost}\\
     & J_1^A - \gamma_a^T\cdot\big(U(\pi_0) \cdot (1-\beta) + U(\pi_1)\cdot\beta) \leq \nonumber \\
     & \qquad J_0^A - \gamma_a^T\cdot\big(U(\pi_0) \cdot (1-\alpha) + U(\pi_1)\cdot\alpha) \label{eqn:icc}\\
     & J_1^A - \gamma_a^T\cdot\big(U(\pi_0) \cdot (1-\beta) + U(\pi_1)\cdot\beta) \leq J_0^A \label{eqn:pc}
\end{align}
where the optimization is over the horizon $T \in \{1,\ldots,\overline{T}\}$ with a constraint on the maximum  horizon $\overline{T} \in \mathbb{Z}_+$, the payment amounts $\pi_0,\pi_1\in\mathbb{R}$ with $\pi_1 \geq \pi_0$, and the decision region $\mathcal{R} \subseteq \mathbb{R}^{T \cdot p}$. The principal minimizes its expected total cost (\ref{eqn:cost}), subject to a constraint (\ref{eqn:icc}) that ensures the contract is such that the agent's expected cost when exerting high effort is lower than the agent's expected cost when exerting low effort (i.e., an \emph{incentive compatibility} constraint that ensures high effort is incentivized) and a constraint (\ref{eqn:pc}) that ensures the agent incurs a lower expected cost from the contract when exerting higher effort than if they exert low effort and receive no payment (i.e., a \emph{participation} constraint that ensures it is rational for the agent to accept the contract). 

\section{STRUCTURE OF OPTIMAL CONTRACT}
\label{sec:soc}

In this section, we theoretically characterize the mathematical structure of the optimal contract.

\subsection{Characterization of Optimal Contract}
The key result is a theorem that relates our contract design problem to the Neyman-Pearson testing framework. The Neyman-Pearson lemma is a foundational result in hypothesis testing that establishes that the likelihood ratio test (LRT) is the most powerful test for distinguishing between two statistical distributions. Specifically, to maximize the probability of correctly identifying high effort ($\beta$) subject to a bound on the probability of falsely identifying low effort as high effort ($\alpha$), the optimal decision region $\mathcal{R}$ must partition the sample space using a threshold on the likelihood ratio $L_1(\mathbf{Y}_T)/L_0(\mathbf{Y}_T)$.


\begin{proposition}
\label{prop:optimality}
Suppose $T,\pi_0,\pi_1$ are fixed with $\pi_1 >  \pi_0$, and assume there exists $\mathcal{R}$ such that
\begin{align}
    \alpha &\leq - U(\pi_0)/(U(\pi_1)-U(\pi_0)) \label{eqn:alph}\\
    \beta &= \Delta - U(\pi_0)/(U(\pi_1)-U(\pi_0)) \label{eqn:beth}
\end{align}
for the constant
\begin{equation}
    \Delta := \gamma_a^{-T}\cdot\frac{J^A_1 - J^A_0}{U(\pi_1) - U(\pi_0)}.\label{eqn:Delta}
\end{equation} Then there exists an optimal decision region $\mathcal{R}^*$ that solves (\ref{eqn:cost})--(\ref{eqn:pc}), which takes the form of
\begin{equation}
\label{eqn:lrt}
    \mathcal{R}^* = \big\{\mathbf{Y}_T : L_1(\mathbf{Y}_T) \geq \kappa \cdot L_0(\mathbf{Y}_T)\big\}.
\end{equation}
where $L_0(\cdot)$ and $L_1(\cdot)$ are the (probabilistic) likelihoods of $\mathbf{Y}_T$ under $H_0$ and $H_1$, respectively, and $\kappa \geq 0$ is a scalar.
\end{proposition}

\begin{proof}
Observe that the only decision variable here is $\mathcal{R}$. We can rewrite the principal's cost (\ref{eqn:cost}) as
\begin{equation}
    J_1^P + \gamma_p^T \cdot \pi_0 + \gamma_p^T\cdot\big(\pi_1 -\pi_0\big)\cdot\beta. \label{eqn:cost2}
\end{equation}
Because $\pi_1 > \pi_0$ and $\gamma_p \in (0,1)$, minimizing (\ref{eqn:cost}) subject to (\ref{eqn:icc}), (\ref{eqn:pc}) is equivalent to minimizing $\beta$ subject to (\ref{eqn:icc}), (\ref{eqn:pc}). Rearranging gives that (\ref{eqn:icc}), (\ref{eqn:pc}) can be rewritten as
\begin{align}
    \beta &\geq \Delta + \alpha \label{eqn:icc2}\\
    \beta &\geq \Delta - U(\pi_0)/(U(\pi_1)-U(\pi_0)) \label{eqn:pc2}
\end{align}
Thus, the principal's problem is equivalent to minimizing $\beta$ subject to (\ref{eqn:icc2}), (\ref{eqn:pc2}). Call this optimization problem $\mathbf{P1}$.

We next consider minimizing $\alpha$ subject to (\ref{eqn:pc2}). Call this optimization problem $\mathbf{P2}$. By the assumptions corresponding to (\ref{eqn:alph}), (\ref{eqn:beth}), the generalized Neyman-Pearson Lemma (see for instance Theorem 3.6.1 of \cite{lehmann2022testing}) implies that (\ref{eqn:lrt}) is optimal for $\mathbf{P2}$ when $\kappa \geq 0$ is chosen such that (\ref{eqn:beth}) holds. (Such a value for $\kappa$ exists because the relevant distributions are continuous.) This means (\ref{eqn:lrt}) is an optimal solution to $\mathbf{P2}$ where (\ref{eqn:alph}), (\ref{eqn:beth}) hold, which implies that this optimal solution to $\mathbf{P2}$ is also an optimal solution to $\mathbf{P1}$.
\end{proof}
\textbf{Remark.}
The assumption that there exists $\mathcal{R}$ satisfying (\ref{eqn:alph})--(\ref{eqn:beth}) is feasible only when the corresponding values of $\alpha$ and $\beta$ lie in $[0,1]$. Without limited liability, a strictly positive $\alpha$ can be achieved by allowing $\pi_0<0$. With limited liability, $\pi_0\geq0$, which together with the feasibility condition implies $\pi_0=0$, and hence $\alpha=0$ and $\beta=\Delta$. $\blacksquare$

The optimal decision region (\ref{eqn:lrt}) corresponds to a likelihood ratio test (LRT). Although the above result assumes $T,\pi_0,\pi_1$ are fixed, our next result shows that the optimal contract still retains this LRT structure even when we also optimize over the decision variables $T,\pi_0,\pi_1$.\\


\begin{theorem}
\label{thm:lrtopt}
Suppose $H_0 \neq H_1$. Then an optimal solution $T^*, \pi_0^*, \pi_1^*, \mathcal{R}^*$ exists for the contract design problem (\ref{eqn:cost})--(\ref{eqn:pc}). The optimal payments are
\begin{align}
\pi_0^* &= \textstyle U^{-1}\Big(\gamma_a^{-T^*}\cdot\frac{-\alpha(\mathcal{R}^*)}{\beta(\mathcal{R}^*) - \alpha(\mathcal{R}^*)}\cdot(J_1^A-J_0^A)\Big)\label{eqn:pi0star}\\
\pi_1^* &= \textstyle U^{-1}\Big(\gamma_a^{-T^*}\cdot\frac{1-\alpha(\mathcal{R}^*)}{\beta(\mathcal{R}^*) - \alpha(\mathcal{R}^*)}\cdot(J_1^A-J_0^A)\Big)\label{eqn:pi1star}
\end{align}
Furthermore, the optimal $\mathcal{R}^*$ takes the form of (\ref{eqn:lrt}) with $T^*$.
\end{theorem}

\begin{proof} 
For the contract design problem, suppose we solve the contract design problem (\ref{eqn:cost})--(\ref{eqn:pc}) by optimizing over the decision variables in the following order:
\begin{equation}
\label{eqn:opt1}
\min_{T\in\mathbb{Z}_+}\,\min_{\mathcal{R}\subseteq\mathbb{R}^{T\cdot p}}\ \min_{\pi_0,\pi_1 : \pi_1 \geq \pi_0}
\end{equation}
The third minimization is interpreted as optimizing $\pi_0,\pi_1$ subject to a fixed $T$ and $\mathcal{R}$. A standard result (Proposition 4.4 of \cite{Laffont}) shows  the solution to this third optimization is 
\begin{figure*}[!t]
\normalsize
\begin{equation}
\label{eqn:uvw}
    U_i = \begin{bmatrix}
C & 0 & \dots & 0 \\
C(A+BK_i) & C & \dots & 0 \\
C(A+BK_i)^2 & C(A+BK_i) & \dots & 0 \\
\vdots & \vdots & \ddots & \vdots \\
C(A+BK_i)^{T-1} & C(A+BK_i)^{T-2} & \dots & C
\end{bmatrix}\qquad  V_i = \begin{bmatrix} C(A+BK_i)^{\hphantom{1}} \\ C(A+BK_i)^2 \\ C(A+BK_i)^3 \\ \vdots \\ C(A+BK_i)^T \end{bmatrix}
\end{equation}
\hrulefill 
\vspace*{4pt}
\end{figure*}
\begin{align}
\pi_0^*(\mathcal{R}) &= \textstyle U^{-1}\Big(\gamma_a^{-T}\cdot\frac{-\alpha(\mathcal{R})}{\beta(\mathcal{R}) - \alpha(\mathcal{R})}\cdot(J_1^A-J_0^A)\Big)\\
\pi_1^*(\mathcal{R}) &= \textstyle U^{-1}\Big(\gamma_a^{-T}\cdot\frac{1-\alpha(\mathcal{R})}{\beta(\mathcal{R}) - \alpha(\mathcal{R})}\cdot(J_1^A-J_0^A)\Big)
\end{align}
Optimality of (\ref{eqn:pi0star}), (\ref{eqn:pi1star}) follows by finishing problem (\ref{eqn:opt1}) by optimizing over $\mathcal{R}$ and then $T$. This optimal solution exists because there exists $\mathcal{R}$ such that $\beta(\mathcal{R}) > \alpha(\mathcal{R})$ since we have $H_0 \neq H_1$ (see for instance Corollary 3.2.1 of \cite{lehmann2022testing}).

Next, let $J(T,\pi_0,\pi_1,\mathcal{R})$ be the objective function (\ref{eqn:cost}) in the contract design problem. We note that
\begin{multline}
    \min_{\mathcal{R}\subseteq\mathbb{R}^{T\cdot p}} \big\{J(T^*,\pi_0^*,\pi_1^*,\mathcal{R}) : (\ref{eqn:icc}), (\ref{eqn:pc}) \big\} \geq \\
\min_{\substack{T\in\mathbb{Z}_+, \mathcal{R}\subseteq\mathbb{R}^{T\cdot p}\\\pi_0,\pi_1 : \pi_1 \geq \pi_0}} \big\{J(T,\pi_0,\pi_1,\mathcal{R}) : (\ref{eqn:icc}), (\ref{eqn:pc})\big\}.
\end{multline}
Since equality occurs when setting $\mathcal{R} = \mathcal{R}^*$ on the left-hand side, this means that solving the optimization problem
\begin{equation}
    \min_{\mathcal{R}\subseteq\mathbb{R}^{T\cdot p}} \big\{J(T^*,\pi_0^*,\pi_1^*,\mathcal{R}) : (\ref{eqn:icc}), (\ref{eqn:pc}) \big\}
\end{equation}
is equivalent to solving the original contract design problem (\ref{eqn:cost})--(\ref{eqn:pc}). This equivalent optimization problem matches the setup of Proposition \ref{prop:optimality}, and so the result follows if we can show the existence of an $\mathcal{R}$ satisfying (\ref{eqn:alph}), (\ref{eqn:beth}). Note that some algebra shows $- U(\pi_0^*)/(U(\pi_1^*)-U(\pi_0^*)) = \alpha(\mathcal{R}^*)$
and for (\ref{eqn:Delta}) that $\Delta = \beta(\mathcal{R}^*) - \alpha(\mathcal{R}^*)$. We finally argue this implies the existence of an $\mathcal{R}$ satisfying (\ref{eqn:alph}), (\ref{eqn:beth}): Indeed, choosing $\mathcal{R} = \mathcal{R}^*$ satisfies (\ref{eqn:alph}), (\ref{eqn:beth}).
\end{proof}

\subsection{Limited Liability Constraints}
In some applications, it can be preferred to impose an additional constraint of the form $\pi_0 \geq 0$ on the contract design problem (\ref{eqn:cost})--(\ref{eqn:pc}). Such a constraint is known as a \emph{limited liability} constraint, and it imposes a requirement that the agent can never receive a negative payment (i.e., the agent can never be fined). Note that without this constraint, Theorem \ref{thm:lrtopt} shows that in general $\pi_0$ will be negative, suggesting that the agent would be penalized for a state trajectory that is likely to occur with low effort.

\begin{theorem}
\label{thm:lrtopt_ll}
Suppose $H_0 \neq H_1$. Then an optimal solution $T^*, \pi_0^*, \pi_1^*, \mathcal{R}^*$ exists for the contract design problem (\ref{eqn:cost})--(\ref{eqn:pc}) with a limited liability constraint $\pi_0 \geq 0$. The optimal payments are given by
\begin{align}
\pi_0^* &= 0\\
\pi_1^* &= \textstyle U^{-1}\Big(\gamma_a^{-T^*}\cdot\frac{1}{\beta(\mathcal{R}^*) - \alpha(\mathcal{R}^*)}\cdot(J_1^A-J_0^A)\Big)\label{eqn:pi1star_ll}
\end{align}
Furthermore, the optimal $\mathcal{R}^*$ has the form of (\ref{eqn:lrt}) with $T^*$.
\end{theorem}

We omit the proof because it is similar to the proof of Theorem \ref{thm:lrtopt}, with the main change being that Proposition \ref{prop:optimality} has to be changed to account for the fact that only the incentive compatibility constraint is binding under limited liability, whereas both the incentive compatibility and participation constraints are binding in the original result.

\subsection{Decision to Induce Effort}
\label{sec:dieff}

After the principal has solved the contract design problem (\ref{eqn:cost})--(\ref{eqn:pc}) with or without limited liability constraints, which assumes the contract induces high effort from the agent, the principal must decide whether it is beneficial from a cost perspective to induce high effort. If the principal declines to induce high effort, then they can simply offer a contract that consists of zero payment (i.e., they can set $\pi_0 = \pi_1 = 0$ with the choice of $T$ and $\mathcal{R}$ irrelevant) \cite{Laffont}. The principal would then incur an expected cost of $J^P_0$. 

The decision to induce effort now depends upon whether the following inequality is satisfied:
\begin{equation}
\label{eqn:eqndieff}
    J_1^P + \gamma_p^{T^*} \cdot \big(\pi_0^* \cdot (1-\beta(\mathcal{R}^*) + \pi_1^* \cdot \beta(\mathcal{R}^*)\big) \leq J_0^P,
\end{equation}
where $T^*,\pi_0^*,\pi_1^*,\mathcal{R}^*$ solve the contract design problem (\ref{eqn:cost})--(\ref{eqn:pc}) with or without liability constraints. The left-hand side of this inequality is the expected cost to the principal under an optimal contract that induces high effort, and the right-hand side is the expected cost to the principal when they offer no payment to the agent. If this inequality holds, then it is optimal to induce high effort, and the principal offers $T^*,\pi_0^*,\pi_1^*,\mathcal{R}^*$ as the contract to the agent. Otherwise, it is optimal to decline to induce high effort, and the principal offers zero payment to the agent.

\section{COMPUTING THE OPTIMAL CONTRACT}
\label{sec:ctoc}
In this section, we provide an approach to design an optimal contract. We derive the likelihood ratio distribution, and  we provide an algorithm to compute the optimal contract.

\subsection{Derivation of Likelihood Ratio Test}
Since Theorems \ref{thm:lrtopt} and \ref{thm:lrtopt_ll} show that the LRT provides an optimal decision region, our next step is to derive the LRT. By a slight abuse of notation, interpret $\mathbf{Y}_T \in \mathbb{R}^{T\cdot p}$ as a vector defined by stacking the observations $y_1,\ldots,y_T$. We first derive the distribution of $\mathbf{Y}_T$.

\begin{proposition}
\label{prop:ytdist}
    Under the state-feedback controller $K_i$ for $i = 0$ or $i = 1$, we have $\mathbf{Y}_T \sim \mathcal{N}(M_i, S_i)$ for $M_i = U_i\boldsymbol{\mu}_W + V_i\mu_0$ and $S_i = U_i\mathbf{\Sigma}_WU_i^\top + V_i\Sigma_0V_i^\top + \mathbf{\Sigma}_E$, with $U_i\in\mathbb{R}^{(T\cdot p)\times(T\cdot n)}$ and $V_i \in \mathbb{R}^{(T\cdot p)\times n}$ given in (\ref{eqn:uvw}), and where $\boldsymbol{\mu}_W \in \mathbb{R}^{T\cdot n}$ is defined by stacking $\mu_w$ a total of $T$ times, $\mathbf{\Sigma}_W = \textrm{blkdiag}(\Sigma_w, \ldots, \Sigma_w) \in \mathbb{R}^{(T\cdot n)\times(T\cdot n)}$ where $\textrm{blkdiag}(\cdot)$ specifies a block diagonal matrix, and $\mathbf{\Sigma}_E = \textrm{blkdiag}(\Sigma_e, \ldots,\Sigma_e) \in \mathbb{R}^{(T\cdot p)\times(T\cdot p)}$.
\end{proposition}

\begin{proof}
    Let $\mathbf{W}_T \in \mathbb{R}^{T\cdot n}$ be defined by stacking the random variables $w_0,\ldots,w_{T-1}$, and let $\mathbf{E}_T \in \mathbb{R}^{T\cdot p}$ be defined by stacking the random variables $\epsilon_0,\ldots,\epsilon_{T-1}$. Under the assumptions in our model, it is well known that $x_{t} = (A+BK_i)^tx_0 + \sum_{\tau = 0}^{t-1}(A+BK_i)^{t-\tau-1}w_\tau$, implying
    \begin{equation}
        \mathbf{Y}_T = U_i\mathbf{W}_T + V_ix_0 + \mathbf{E}_T.\label{eqn:lincom}
    \end{equation} 
    Since $x_0, \mathbf{W}_T, \mathbf{E}_T$ are jointly Gaussian and independent, they are all together jointly Gaussian. Hence, any linear combination such as (\ref{eqn:lincom}) is also jointly Gaussian. This means we completely know the distribution of $\mathbf{Y}_T$ if we calculate its mean and covariance. By linearity of expectation, we have $\mathbb{E}(\mathbf{Y}_T) = U_i\mathbb{E}(\mathbf{W}_T) + V_i\mathbb{E}(x_0) + \mathbb{E}(\mathbf{E}_T) = 
        U_i\boldsymbol{\mu}_W + V_i\mu_0$. If we denote the covariance using the notation $\textrm{Cov}(\cdot)$, then $\textrm{Cov}(\mathbf{Y}_T) = \textrm{Cov}(U_i\mathbf{W}_T) + \textrm{Cov}(V_ix_0) + \textrm{Cov}(\mathbf{E}_T) = 
        U_i\mathbf{\Sigma}_WU_i^\top + V_i\Sigma_0V_i^\top + \mathbf{\Sigma}_E$, where we used independence of $x_0,\mathbf{W}_T,\mathbf{E}_T$ in the first equality. The result follows.
\end{proof}

Because jointly Gaussian random variables belong to the exponential family, it is easier to implement the LRT using the log-likelihood ratio. Our next result characterizes the distribution of the log-likelihood ratio.

\begin{theorem}
\label{thm:distribution}
Let $L_0(\mathbf{Y}_T)$ be the likelihood of $\mathbf{Y}_T$ under $H_0$, and let $L_1(\mathbf{Y}_T)$ be the likelihood of $\mathbf{Y}_T$ under $H_1$. The log-likelihood ratio $LLR = \log L_1(\mathbf{Y}_T) - \log L_0(\mathbf{Y}_T)$ follows a generalized chi-squared distribution, expressed as a linear combination of independent non-central chi-squared variables and a standard normal random variable:
\begin{align}
    LLR \mid H_0 &\sim \textstyle\sum_{j \in \mathcal{K}_0} w_{j} \chi^2_1(\nu_{j}^2) + \sigma_{0} Z_0 + C_0, \\
    LLR \mid H_1 &\sim \textstyle\sum_{j \in \mathcal{K}_1} \bar{w}_{j} \chi^2_1(\bar{\nu}_{j}^2) + \sigma_{1} Z_1 + C_1,
\end{align}
where $Z_0, Z_1 \sim \mathcal{N}(0,1)$, and $\mathcal{K}_0, \mathcal{K}_1$ are the index sets of dimensions with non-zero weights. The distribution parameters are determined by the mean shift $D = M_1 - M_0$ and the covariance matrices $S_0, S_1$. Define the core symmetric matrices with their respective eigenvalue decompositions: $\Phi_0 = S_0^{1/2} S_1^{-1} S_0^{1/2} = \Psi_0 \Lambda \Psi_0^\top$ and $\Phi_1 = S_1^{1/2} S_0^{-1} S_1^{1/2} = \Psi_1 \bar{\Lambda} \Psi_1^\top$. The matrices $\Psi_0, \Psi_1 \in \mathbb{R}^{(T\cdot p) \times (T\cdot p)}$ are orthogonal matrices whose columns consist of the eigenvectors of $\Phi_0$ and $\Phi_1$, respectively. $\Lambda, \bar{\Lambda} \in \mathbb{R}^{(T\cdot p) \times (T\cdot p)}$ are the diagonal matrices containing the corresponding eigenvalues.
The generalized chi-squared weights and rotated linear coefficients are derived directly from these eigenvalues and eigenvectors: $w_j = (1 - \lambda_j)/2$ and $\bar{w}_j = (\bar{\lambda}_j - 1)/2$ with $\lambda_j$ and $\bar{\lambda}_j$ being the diagonal elements of $\Lambda$ and $\bar{\Lambda}$. The other parameters are: $b_0 = \Psi_0^\top S_0^{1/2} S_1^{-1} D$, $b_1 = \Psi_1^\top S_1^{1/2} S_0^{-1} D$, $nu_{j} = \frac{b_{0,j}}{2w_{j}}$, $\bar{\nu}_{j} = \frac{b_{1,j}}{2\bar{w}_{j}}$, $\sigma_0^2 = \sum_{j \notin \mathcal{K}_0} b_{0,j}^2$, $\sigma_1^2 = \sum_{j \notin \mathcal{K}_1} b_{1,j}^2$, $C_0 = ( \log |S_{0}|/|S_{1}| - D^\top S_{1}^{-1} D )/2 - \sum_{j \in \mathcal{K}_0} w_{j} \nu_{j}^2$, and $C_1 =( \log |S_{0}|/|S_{1}| + D^\top S_{0}^{-1} D )/2 - \sum_{j \in \mathcal{K}_1} \bar{w}_{j} \bar{\nu}_{j}^2$.
\end{theorem}

\begin{proof}
The log-likelihood ratio simplifies to $LLR = \big(\log |S_{0}|/|S_{1}| + (\mathbf{Y}_T - M_0)^\top S_0^{-1}(\mathbf{Y}_T - M_0) + 
    - (\mathbf{Y}_T - M_1)^\top S_1^{-1}(\mathbf{Y}_T - M_1) \big)/2$. By Proposition \ref{prop:ytdist}, the stacked observation vector is distributed as $\mathbf{Y}_T \sim \mathcal{N}(M_i, S_i)$ under hypothesis $i \in \{0, 1\}$. We determine the distribution of the log-likelihood ratio by expanding it into a multivariate quadratic form and applying the exact standardization and diagonalization method detailed in Das and Geisler \cite{das2020methods}.

\vspace{1ex}
\noindent \textbf{Case 1: Under the Null Hypothesis ($H_0$).} \\
Here, $\mathbf{Y}_T \sim \mathcal{N}(M_0, S_0)$. Following \cite{das2020methods}, we first apply the whitening transformation $\mathbf{z} = S_0^{-1/2}(\mathbf{Y}_T - M_0) \sim \mathcal{N}(0, I)$. This transforms $S$ to a new quadratic form $\tilde{q}(\mathbf{z}) = \mathbf{z}^\top \tilde{Q}_0 \mathbf{z} + \tilde{q}_0^\top \mathbf{z} + c_{0}$, where the standardized coefficients are computed as $\tilde{Q}_0 = (I - S_0^{1/2} S_1^{-1} S_0^{1/2})/2 = (I - \Phi_0)/2$, $\tilde{q}_0 = S_0^{1/2} S_1^{-1} D$, and $c_{0} = ( \log |S_0|/|S_1| - D^\top S_1^{-1} D )/2$. Because $\Phi_0$ is orthogonally diagonalizable as $R_0 \Lambda R_0^\top$, we can substitute this to decompose $\tilde{Q}_0$ as $\tilde{Q}_0 = (I - \Psi_0 \Lambda \Psi_0^\top)/2 = \Psi_0 [ (I - \Lambda)/2 ] \Psi_0^\top \equiv \Psi_0 W \Psi_0^\top$, where $W$ is diagonal with  $w_j = (1 - \lambda_j)/2$. We define the rotated standard normal vector $\mathbf{y} = \Psi_0^\top \mathbf{z} \sim \mathcal{N}(0, I)$. Substituting $\mathbf{z} = \Psi_0 \mathbf{y}$ decouples the quadratic form into a sum of independent scalar variables:
\begin{align}
    \tilde{q}(\mathbf{y}) &= (\Psi_0 \mathbf{y})^\top (\Psi_0 W \Psi_0^\top) (\Psi_0 \mathbf{y}) + \tilde{q}_0^\top (\Psi_0 \mathbf{y}) + c_{0} \nonumber \\
    &= \mathbf{y}^\top W \mathbf{y} + b_0^\top \mathbf{y} + c_{0} \nonumber \\
    &= \textstyle\sum_{j=1}^{T\cdot p} \left( w_j y_j^2 + b_{0,j} y_j \right) + c_{0},
\end{align}
where $b_0 = \Psi_0^\top \tilde{q}_0$. We now partition this sum into dimensions with non-zero weights ($j \in \mathcal{K}_0$) and zero weights ($j \notin \mathcal{K}_0$). For dimensions where $w_j \neq 0$, we complete the square: $w_j y_j^2 + b_{0,j} y_j = w_j ( y_j + b_{0,j}/(2w_j) )^2 - b_{0,j}^2/(4w_j) \sim w_j \chi^2_1(\nu_j^2) - w_j \nu_j^2$, where $\nu_j = b_{0,j}/(2w_j)$ defines the non-centrality parameter. For dimensions where $w_j = 0$, the quadratic term vanishes. The remaining linear terms aggregate into a single standard normal distribution: $\sum_{j \notin \mathcal{K}_0} b_{0,j} y_j \sim \mathcal{N}(0, \sum_{j \notin \mathcal{K}_0} b_{0,j}^2) \equiv \sigma_0 Z_0$. Finally, substituting these independent random variables back into the sum and absorbing the constants yields the distribution for $H_0$: $LLR \mid H_0 \sim \sum_{j \in \mathcal{K}_0} w_j \chi^2_1(\nu_j^2) + \sigma_0 Z_0 + c_{0} - \sum_{j \in \mathcal{K}_0} w_j \nu_j^2$, where we get $C_0 = c_{0} - \sum_{j \in \mathcal{K}_0} w_j \nu_j^2$.

\vspace{1ex}
\noindent \textbf{Case 2: Under the Alternative Hypothesis ($H_1$).} \\
By symmetry, under $H_1$, $\mathbf{Y}_T \sim \mathcal{N}(M_1, S_1)$. The standardizing transformation $\bar{\mathbf{z}} = S_1^{-1/2}(\mathbf{Y}_T - M_1)$ maps $S$ into the quadratic form $\bar{q}(\bar{z}) = \bar{\mathbf{z}}^\top \tilde{Q}_1 \bar{\mathbf{z}} + \tilde{q}_1^\top \bar{\mathbf{z}} + c_{1}$ where the quadratic coefficients are defined as $\tilde{Q}_1 = (S_1^{1/2} S_0^{-1} S_1^{1/2} - I)/2 = (\Phi_1 - I)/2$, $\tilde{q}_1 = S_1^{1/2} S_0^{-1} D$, and $c_{1} = ( \log |S_0|/|S_1| + D^\top S_0^{-1} D )/2$. Applying the same direct decomposition and rotation decoupling as in the $H_0$ case algebraically produces the alternative weights $\bar{w}_j$, non-centralities $\bar{\nu}_j$, normal variance $\sigma_1^2$, and offset $C_1$, which concludes the proof.
\end{proof}

\begin{table}[t]
    \centering
    \caption{Mathematical Assumptions and Simulation Parameters}
    \label{tab:math_assumptions}
    \renewcommand{\arraystretch}{1.05} 
    \begin{tabular}{lcc}
        \hline
        \textbf{Parameter} & \textbf{Ex. 1: LFC} & \textbf{Ex. 2: Wellness} \\
        \hline
        $x_{k+1}$ & $Ax_k + Bu_k + \omega_k$ & $a x_k + b u_k + \omega_k$ \\
        $y_k$ & $Cx_k + \epsilon_k$ & $c x_k + \epsilon_k$ \\
        $x_k$ & $[x_g, x_t, x_p, x_i]^\top$ & $\text{scalar weight}$ \\
        $C$ & $I_{4 \times 4}$ & $1$ \\
        $x_0$ & $\mathcal{N}([0,0,0.1,0]^\top, 10^{-9}\cdot I)$ & $\mathcal{N}(65, 0)$\\
        $w_k$ & $\mathcal{N}(0, 0.01\cdot I)$ & $\mathcal{N}(0.817, 0.5)$ \\
        $\epsilon_k$ & $\mathcal{N}(0, 0.01 \cdot I)$ & $\mathcal{N}(0, 0.01)$ \\
        $u_0$ & $K_0 x_k (\lambda = 0.2)$ & $\delta_0x_k(\delta_0 = 70)$ \\
        $u_1$ & $K_1 x_k (\lambda = 0.05)$ & $\delta_1x_k(\delta_1 = 110)$  \\
        $J^A_1 - J^A_0$ & 0.1 & 5\\
        Time Step & $0.1s$ & 1 week \\
        \hline
    \end{tabular}
\end{table}
\subsection{Computational Algorithm}

Our characterization of the optimal contract structure in Section \ref{sec:soc} is significant because  it allows us to convert solving the contract design problem (\ref{eqn:cost})--(\ref{eqn:pc}), which is a complex, infinite dimensional-optimization problem, into two line searches, one of which is nested within the other.

For the contract design problem (\ref{eqn:cost})--(\ref{eqn:pc}), define the decision region $R(T,\eta) = \{LLR \geq \eta\}$ where $LLR$ is as defined in Theorem \ref{thm:distribution}, and define the payments
\begin{align}
\pi_0(T,\eta) &= \textstyle U^{-1}\Big(\gamma_a^{-T}\cdot\frac{-\alpha(\mathcal{R}(T,\eta))}{\beta(\mathcal{R(T,\eta)}) - \alpha(\mathcal{R}(T,\eta))}\cdot(J_1^A-J_0^A)\Big)\\
\pi_1(T,\eta) &= \textstyle U^{-1}\Big(\gamma_a^{-T}\cdot\frac{1-\alpha(\mathcal{R}(T,\eta))}{\beta(\mathcal{R(T,\eta)}) - \alpha(\mathcal{R(T,\eta)})}\cdot(J_1^A-J_0^A)\Big)
\end{align}
Then we can solve the contract design problem (\ref{eqn:cost})--(\ref{eqn:pc}) using the following algorithm:
\begin{enumerate}
    \item For $t \in \{1,\ldots,\overline{T}\}$,
    \begin{enumerate}
        \item For $\eta \in \mathbb{R}$, set $J(T,\eta) := J_1^P + \gamma_p^{T} \cdot (\pi_0(T,\eta) \cdot (1-\beta(\mathcal{R}(T,\eta)) + \pi_1(T,\eta) \cdot \beta(\mathcal{R}(T,\eta))$;
        \item Set $\eta^*(T) := \arg\min_{\eta \in \mathbb{R}}J(T,\eta)$;
    \end{enumerate}
    \item Set $T^* := \arg\min_{t\in\{1,\ldots,\overline{T}\}} J(T,\eta^*(T))$;
    \item Set $\eta^* := \eta^*(T^*)$, set $\pi_0^* := \pi_0(T^*,\eta^*)$, and set $\pi_1^* := \pi_1(T^*,\eta^*)$;
\end{enumerate}
Though steps 1.a and 1.b are infinite-dimensional, numerically we can consider a bounded range for $\eta$ and perform the calculations at a finite number of points within this bounded range. Also, note that $\alpha,\beta$ can be computed using Imhof's method, via Davies's algorithm \cite{Imhof1961, Davies1980, Das2025}. In our experiments detailed in the next section, we use the Python package \texttt{chi2comb} \cite{horta2021chi2comb} for the computation.

The contract design problem (\ref{eqn:cost})--(\ref{eqn:pc}) with a limited liability constraint $\pi_0 \geq 0$ can be solved with the algorithm given above but with the change that we use
\begin{align}
\pi_0(T,\eta) &= 0\\
\pi_1(T,\eta) &= \textstyle U^{-1}\Big(\gamma_a^{-T}\cdot\frac{1}{\beta(\mathcal{R(T,\eta)}) - \alpha(\mathcal{R(T,\eta)})}\cdot(J_1^A-J_0^A)\Big)
\end{align}
as our payment amounts in the algorithm.


\section{APPLICATION EXAMPLES}
\label{sec:ae}

\begin{figure}[t]
    \centering
    \includegraphics[width=0.45\textwidth]{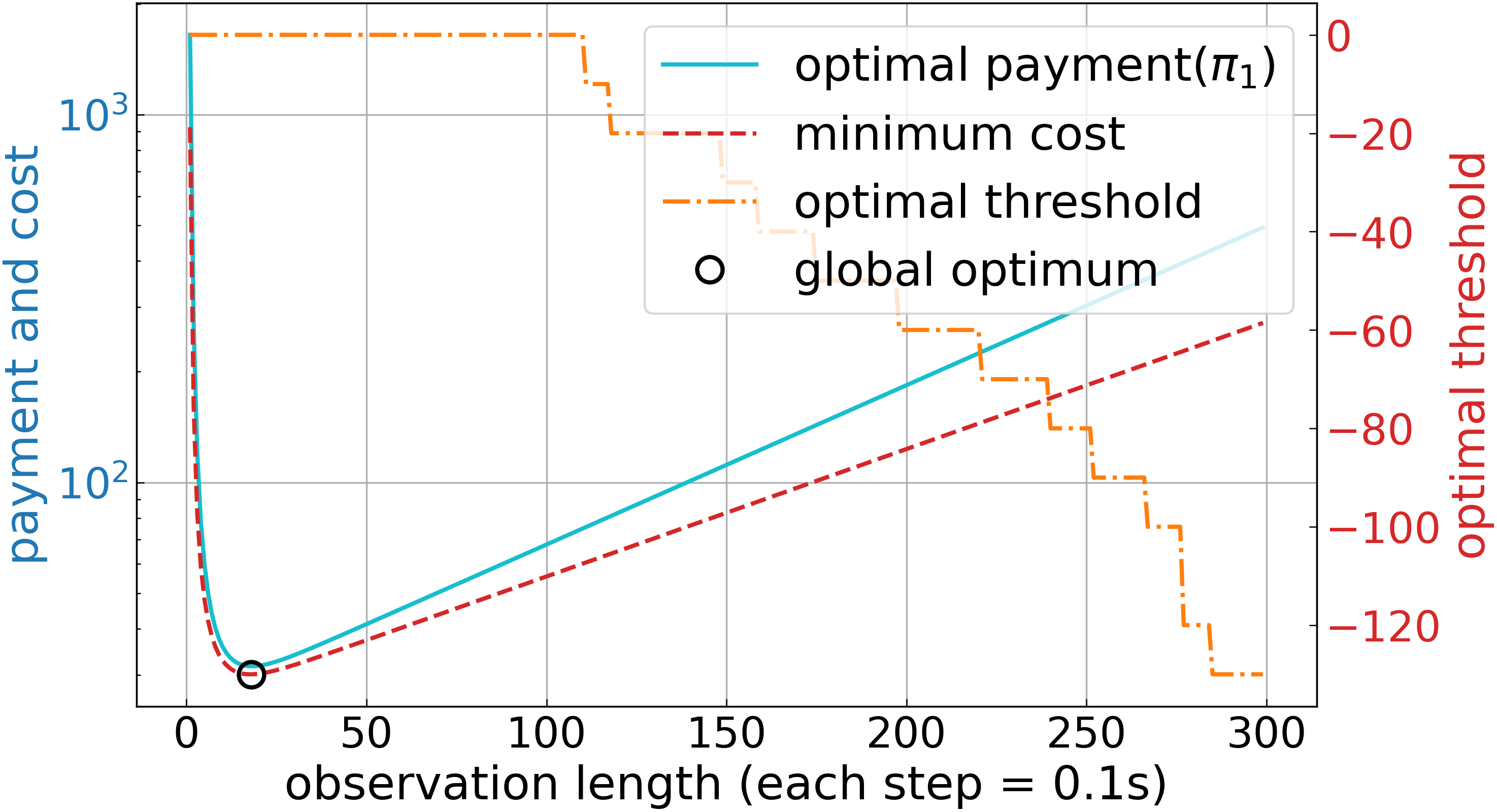}
    \caption{Optimal Contract Parameters for LFC Example.}
    \label{fig:optimal_lfc}
\end{figure}

We apply our framework to applications from power systems and health care. We focus on the contract design problem (\ref{eqn:cost})--(\ref{eqn:pc}) with limited liability constraints (i.e., the agent will not accept being fined). In both examples, we use $U(\pi) = \sqrt{\pi}$ as the agent's utility function, which satisfies $U' > 0$ and $U'' < 0$ for nonnegative payments. We interpret $\pi$ as having units of US Dollars. We do not consider whether to induce effort as discussed in Section \ref{sec:dieff}, because this decision uses a simple inequality (\ref{eqn:eqndieff}) and because the answer is highly sensitive to the choice for the value of $J_1^P-J_0^P$.

In both examples, we also assume that the agent's discount factor is strictly lower than the principal's (i.e., $\gamma_a < \gamma_p$). The assumption is intended to capture situations in which the principal, an institutional entity such as an Independent System Operator (ISO) or a health insurance company, has a lower effective discount rate (equivalently, a higher discount factor) than the individual agent. This assumption affects the optimal contract horizon but is not required for our theoretical characterization of the optimal contract. In our simulations, we set $\gamma_a = 0.995$ and $\gamma_p = 0.998$.

 
\subsection{Load Frequency Control (LFC) Model}
Load frequency control (LFC) \cite{Tan2010} is important for stabilizing the frequency in a power system. Here, the principal (e.g., the ISO) offers a performance-based contract to incentivize the agents (e.g., power generators) to adopt a high-performance control policy ($K_1$) for frequency regulation. When the ISO cannot directly inspect the generator's internal control settings, a moral hazard problem arises. The contract design we suggest enables the principal to infer the internal settings of the generator from a partial, noisy history of the system states, thereby incentivizing the implementation of high-performance control.

We discretize a single-area LFC system with a sampling period of $T = 0.1s$ using Zero-Order Hold (ZOH) \cite{Tan2010}. The state-space distributions and system variables are summarized in Table \ref{tab:math_assumptions}. Based on benchmark parameters $T_G = 0.08, T_T = 0.3, T_P = 20, K_P = 120$, and $R =2.4$, the discrete-time system matrices are
\begin{equation*}
A = \begin{bmatrix}
0.287 & 0 & 0 & 0 \\
0.156 & 0.717 & 0 & 0 \\
0.061 & 0.509 & 0.995 & 0 \\
0.002 & 0.027 & 0.100 & 1
\end{bmatrix}\qquad 
B = \begin{bmatrix}
0.713 \\ 0.127 \\ 0.029 \\ 0.001
\end{bmatrix}.
\end{equation*}
We simplify the IMC-PID tuning from \cite{Tan2010} for the agent's controller: By omitting the disturbance rejection filter $Q_d(s)$, we choose $K_0$ (low effort) and $K_1$ (high effort) by setting the closed-loop time constant $\lambda$ to $0.2$ and $0.05$, respectively. This results in $K_0 = [0, 0.0002, -0.0856, -0.0139]$ and $K_1 = [0, -2.2033, -0.6932, -0.0556]$. 

Figure \ref{fig:optimal_lfc} shows the optimal payment $\pi_1^*$, decision threshold $\ln \eta^*$, and expected cost as a function of the contract length $T$. We assume the agent maintains a chosen control setting over a physical operating cycle of 30 seconds, with the possibility of switching their controller every 30 seconds. Accordingly, we use a maximum observation horizon of 30 seconds ($\overline{T} = 300$). The globally optimal observation length is attained at $T^* = 18$ ($1.8$ seconds), where the principal’s expected cost is minimized at $30.0923$, requiring an optimal incentive payment of $\pi_1^* = \$31.6495$.


\subsection{Wellness Incentive Program}
Imagine a wellness program where a health insurance company provides payments to encourage weight loss over 200 weeks. The Mifflin-St Jeor equation \cite{mifflin1990new,aswani2019behavioral} gives the coefficients of an affine LTI system to describe weight $x_t$ (kg) dynamics under modest amounts of weight change, where the coefficients depend upon an individual's age, starting weight, height, and gender. For a 30-year-old female ($h = 160.0$ cm, $x_0 = 65.0$ kg), the dynamics with a time step of one-week are $    x_{t+1} = 0.999 x_t -6.429 u_t + 1.286\times 10^{-4} f_t-0.618 + z_t$,  where $z_t$ is a random (biologically-driven) fluctuation, $f_t$ is caloric intake (we set $f_t \equiv 1600$ kcal), and $u_t$ is physical activity in units of steps. We modeled the agent's effort as a state-dependent feedback control $u_t = \delta_i x_t$, where $\delta_1 = 110$ represents a high-effort regimen and $\delta_0 = 70$ represents baseline activity. For example, for a person who weighs $65.0$ kg, $\delta_0 = 70$ means the person would take 4,550 steps a day, and $\delta_1 = 110$ means they would take 7,150 steps a day. Figure \ref{fig:optimal_wellness} illustrates the trade-offs involved in designing the wellness incentive program. The global minimum of the insurance company's cost occurs with a contract length of 80 weeks, where the cost is 121.58 and optimal $\pi_1$ is \$217.02.

\begin{figure}[t]
    \centering
    \includegraphics[width=0.45\textwidth]{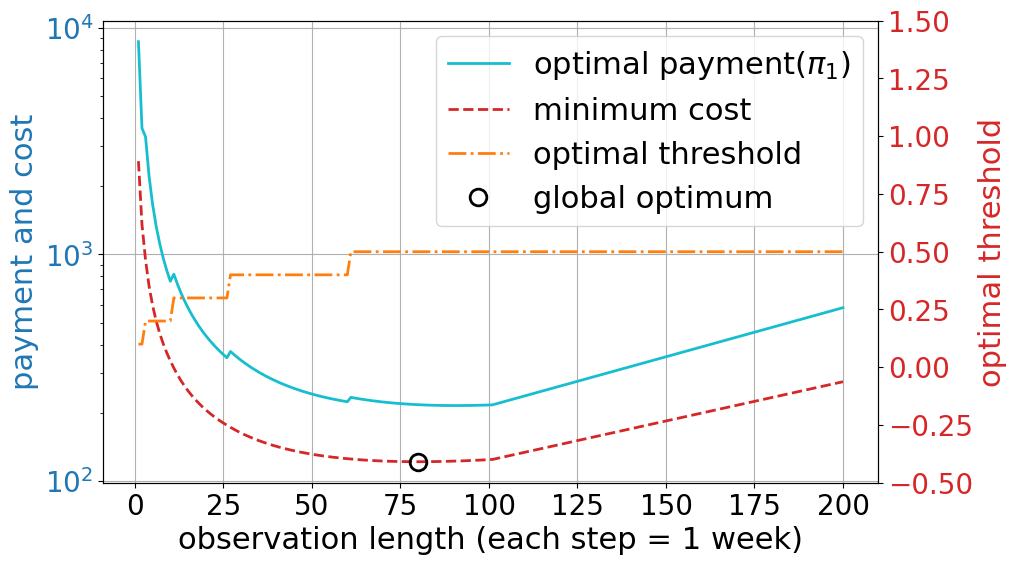}
    \caption{Optimal Contract Parameters for Wellness Example.}
    \label{fig:optimal_wellness}
\end{figure}


\subsection{Discussions}
Despite the different physical contexts of the two examples, our simulations reveal a functionally identical trade-off in the contract design. Figures \ref{fig:optimal_lfc} and \ref{fig:optimal_wellness} illustrate how the principal's expected cost and the optimal high-effort payment ($\pi_1^*$) evolve across varying contract evaluation horizons ($T$). In both scenarios, short evaluation horizons force the principal to incur higher expected costs. Over a short horizon, system stochasticity dominates the observed state trajectories, forcing the principal to offer a substantially larger payment $\pi_1^*$ to overcome uncertainty and make the contract feasible. As $T$ increases, the structural divergence between the high-effort and low-effort policies begins to overpower the noise ($w_k$, $\epsilon_k$). This effectively reduces the informational friction between the principal and the agent, reducing the required statistical risk premium. After a certain point, this statistical benefit is counterbalanced by the cost of delayed compensation driven by the discount factor asymmetry ($\gamma_a < \gamma_p$). Specifically, because we assume a square-root utility function for the agent, the optimal high-effort payment $\pi_1^*$ and the principal's expected cost scale with the exponential multiplier $(\gamma_p / \gamma_a^2)^T$. As this multiplier grows with $T$, it creates the observed U-shaped cost curves.



\section{CONCLUSION}

We studied incentive design for stochastic LTI systems with moral hazard. When the agent chooses between two controllers, we showed that the optimal contract with two payment levels uses a likelihood ratio test to determine payments. We finished by conducting numerical experiments with two examples, which demonstrated that the optimal contract evaluation horizon is governed by a trade-off of extending the observation window, which allows a statistically more powerful test, and the cost of delayed compensations. Future work involves generalizing our results to nonlinear dynamics and controllers and non-quadratic costs.

\bibliographystyle{IEEEtran}
\bibliography{main}

\end{document}